\documentclass[11pt]{amsart}
\usepackage[english]{babel}
\usepackage{amsfonts,latexsym,amsthm,amssymb,graphicx}
\usepackage{mathabx}
\usepackage[all]{xy}
\usepackage[usenames]{color}
\usepackage{tikz}
\usetikzlibrary{shapes}
\usepackage{xypic}
\usetikzlibrary{decorations.pathreplacing}
\usepackage{amsmath}
\usepackage{graphicx}
\usepackage[enableskew]{youngtab}
\usepackage[utf8]{inputenc}
\usepackage[english]{babel}
\usepackage{tikz}
\usetikzlibrary{arrows}
\usepackage{float}
\usepackage{pgf}
\usepackage{subcaption}
\usepackage{pgfplots}
\usepackage{amscd}
\usepackage{tikz}
\usepgflibrary{decorations.shapes}
\usetikzlibrary{shapes.geometric, decorations.shapes,decorations.markings, shapes, fit, arrows, positioning, trees, mindmap, calc}
\usepackage[colorlinks=true, linkcolor=blue, citecolor=blue, urlcolor=blue]{hyperref}

\DeclareMathOperator{\Gr}{Gr}

\newcommand{\Z}{{\mathbb Z}}

\newcommand{\cO}{{\mathcal O}}

\CompileMatrices

\newtheorem{thm}{Theorem}[section]
\newtheorem{lemma}[thm]{Lemma}

\newtheorem{prop}[thm]{Proposition}

{\theoremstyle{defn} \newtheorem{defn}[thm]{Definition}}

{\theoremstyle{remark} \newtheorem{remark}[thm]{Remark}
}

\begin{document}

\title{Galkin's Lower Bound Conjecture Holds for the Grassmannian} 
\author{La'Tier Evans} 
\address{ Department of Mathematics and Computer Science, Henson Science Hall, Salisbury University, Salisbury MD 21801 USA}\email{levans4@gulls.salisbury.edu}
\author{Lisa Schneider} 
\address{ Department of Mathematics and Computer Science, Henson Science Hall, Salisbury University, Salisbury MD 21801 USA}\email{lmschneider@salisbury.edu}
\author{Ryan M. Shifler} 
\address{ Department of Mathematics and Computer Science, Henson Science Hall, Salisbury University, Salisbury MD 21801 USA}\email{rmshifler@salisbury.edu}
\author{Laura Short} 
\address{ Department of Mathematics and Computer Science, Henson Science Hall, Salisbury University, Salisbury MD 21801 USA}\email{lshort1@gulls.salisbury.edu}
\author{Stephanie Warman} 
\address{ Department of Mathematics and Computer Science, Henson Science Hall, Salisbury University, Salisbury MD 21801 USA}\email{swarman2@gulls.salisbury.edu}

\begin{abstract} 
Let $\Gr(k,n)$ be the Grassmannian. The quantum multiplication by the first Chern class $c_1(\Gr(k,n))$ induces an endomorphism $\hat c_1$ of the finite-dimensional vector space $\mathrm{QH}^*(\Gr(k,n))_{|q=1}$ specialized at $q=1$. Our main result is a case that a conjecture by Galkin holds. It states that the largest real eigenvalue of $\hat{c}_1$ is greater than or equal to $\dim \Gr(k,n)$+1 with equality if and only if $\Gr(k,n)=\mathbb{P}^{n-1}.$ \end{abstract}
\maketitle

\section{Introduction}\label{s:intro} Fix $1 \leq k \leq n-1$ and let $\Gr(k,n)$ be the Grassmannian. This is the parametrization of $k$ dimensional linear spaces $V \subset \mathbb{C}^n$. The purpose of this paper is to verify a conjecture by Galkin \cite{Galkin} for the Grassmannian. We recall the precise statements, following \cite{Galkin} and \cite[\S 3]{GGI}.

Let $F$ be a Fano manifold. Let $K:=K_{F}$ be the canonical bundle of $F$ and let $c_1(F):=c_1(-K) \in H^2(F)$ be the anticanonical class. The quantum cohomology ring $(\mathrm{QH}^*(F), \star)$ is a graded algebra over $\Z[q]$, where $q$ is the quantum parameter. Consider the specialization $H^\bullet(F):= \mathrm{QH}^*(F)_{|q=1}$ at $q=1$. The quantum multiplication by the first Chern class $c_1(F)$  induces an endomorphism $\hat c_1$ of the finite-dimensional vector space $H^\bullet(F)$:
 \[ y\in H^\bullet(F) \mapsto \hat c_1(y):= (c_1(F)\star y)|_{q=1} \/. \] Denote by $\delta_0:=\max\{|\delta|:\delta \mbox{ is an eigenvalue of } \hat c_1\}$. Galkin's lower bound conjecture \cite{Galkin} states the following:
 \[ \delta_0 \geq \dim_{\mathbb{C}}F+1 \mbox{ with equality if and only if } F=\mathbb{P}^N. \]
 The conjecture was verified for del Pezzo surfaces \cite{HKLY} and projective complete intersections \cite{Ke}. 
\begin{remark}
The conjecture is related to Property $\mathcal{O}$ which states the following:
  \begin{enumerate}
    \item The real number $\delta_0$ is an eigenvalue of $\hat c_1$ of multiplicity one.
    \item If $\delta$ is any eigenvalue of $\hat c_1$ with $|\delta|=\delta_0$, then $\delta=\delta_0 \zeta$ for some $r$-th root of unity $\zeta\in \mathbb{C}$, where $r$ is the Fano index of $F$.
  \end{enumerate}

Property $\cO$ is conjectured to hold for any complex Fano manifold $F$. Property $\cO$ has been verified for the Grassmannian and several other cases of Fano manifolds; see \cite{Riet,GaGo,Cheo,ChLi,GaIr,SaSh,Ke,HKLY,LMS,BFSS,Withrow}. It is the main hypothesis needed for the statement of Gamma Conjectures I and II, which in turn are related to mirror symmetry on $F$ and generalize the Dubrovin conjectures; we refer the reader to \cite{GGI} for details. 

\end{remark}

The (small) quantum cohomology of $\Gr(k,n)$, denoted by $\mbox{QH}^*(\Gr(k,n))$, has a graded $\mathbb{Z}[q]$-basis consisting of Schubert classes $\{\sigma_\lambda \}$ indexed by partitions in the set \[ \Lambda:=\{(\lambda_1 \geq \lambda_2 \geq \cdots \geq \lambda_k): n-k \geq \lambda_1, \lambda_k \geq 0\} \] where $|\lambda|=\lambda_1+\lambda_2+\cdots+\lambda_k$ for $\lambda \in \Lambda$ and $\deg q=n$. The ring multiplication is given by \[ \sigma_\lambda \star \sigma_\mu = \sum_{\nu, d \geq 0} c_{\lambda, \mu}^{\nu,d}q^d \sigma_\nu \] where $c_{\lambda, \mu}^{\nu,d}$ is the Gromov-Witten invariant that enumerates degree $d$ rational curves intersecting general translates of $\sigma_{\lambda}, \sigma_{\mu},$ and the Poincar\'e dual of $\sigma_{\nu}$. We refer the reader to \cite{FP}, \cite{buch}, and \cite{kontsevich.manin:GW:qc:enumgeom} for additional details regarding the quantum cohomology of the Grassmanian. The quantum Chevalley formula is given by the multiplication of the divisor class $\sigma_{(1)}$ (\cite{bertram}). Let $\lambda^*=(\lambda_2-1 \geq \cdots \geq \lambda_k-1 \geq 0)$ if $\lambda_1=n-k$ and $\lambda_k>0$, otherwise $\lambda^*$ does not exist.

If $\lambda \in \Lambda$ then
\[\sigma_{(1)} \star \sigma_{\lambda}=q \sigma_{\lambda^*}+\sum \sigma_{\mu} \]
where the sum is over all partitions $\mu \in \Lambda$ such that $|\mu|=|\lambda|+1$ and $\lambda \subset \mu$ and the quantum term is omitted if $\lambda^*$ does not exist. The anticanonical class is $c_1(\Gr(k,n))=n \sigma_{(1)}$ and the Fano index is $n$. The linear operator $\hat c_1$ is given next. 

If $\lambda \in \Lambda$ then
\[ \hat c_1( \sigma_{\lambda}) =n\sigma_{\lambda^*}+n\sum \sigma_{\mu} \]
where the sum is over all partitions $\mu \in \Lambda$ such that $|\mu|=|\lambda|+1$ and $\lambda \subset \mu$ and the term $\sigma_{\lambda^*}$ is omitted if $\lambda^*$ does not exist.

\begin{remark}
There is a graph theoretic context for $\delta_0$. We will now recall the notion of the (oriented) {\it quantum Bruhat graph} $\mathcal{G}$ of $\Gr(k,n)$; see \cite{brenti.fomin}. The vertices of this graph consist of partitions  $\lambda \in \Lambda$. There is an oriented edge $\lambda \rightarrow \mu$ if the class $\sigma_\mu$ appears with a positive coefficient (possibly involving $q$) in the quantum Chevalley multiplication $\sigma_{(1)} \star \sigma_{\lambda}$. See Figure \ref{graphs} for examples. By the Perron-Frobenius Theory of nonnegative matrices, $\delta_0$ is equal to the largest real eigenvalue of $n$ times the incidence matrix of $\mathcal{G}$; see \cite{Riet,ChLi}. 
\end{remark}

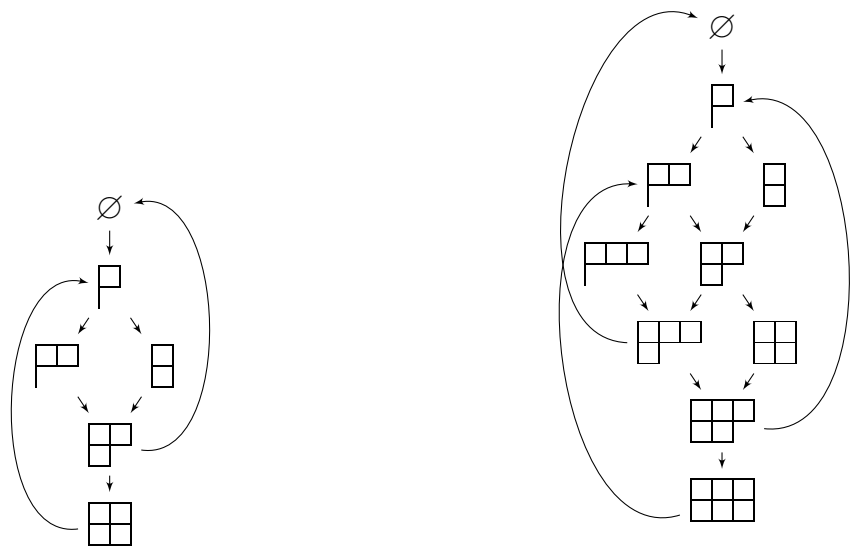
\begin{figure}
\begin{subfigure}[b]{.49\linewidth}
\centering
\begin{tikzpicture}[scale=.7]
\tikzset{edge/.style = {->,> = latex'}}
    \node (0) at (0,0) {$\tiny\emptyset$};
    \node (1) at (0,-1.5) {$\tiny\yng(1,0)$};
    \node (2) at (-1,-3) {$\tiny\yng(2,0)$};
    \node (3) at (1,-3) {$\tiny\yng(1,1)$};
    \node (4) at (0,-4.5) {$\tiny\yng(2,1)$};
    \node (5) at (0,-6) {$\tiny\yng(2,2)$};
  
    \draw [edge] (0) to (1);
    \draw [edge] (1) to (2);
    \draw [edge] (1) to (3);
    \draw [edge] (2) to (4);
    \draw [edge] (3) to (4);
    \draw [edge] (4) to (5);
    \draw [edge] (4) to [bend right=100] (0); 
    \draw [edge] (5) to [bend left=100](1); 
\end{tikzpicture}
\caption{Quantum Bruhat graph for $\Gr(2,4)$}\label{fig1a}
\end{subfigure}\hfill
\begin{subfigure}[b]{.49\linewidth}
\centering
\begin{tikzpicture}[scale=.7]
\tikzset{edge/.style = {->,> = latex'}}
    \node (0) at (0,0) {$\tiny\emptyset$};
    \node (1) at (0,-1.5) {$\tiny\yng(1,0)$};
    \node (2) at (-1,-3) {$\tiny\yng(2,0)$};
    \node (3) at (1,-3) {$\tiny\yng(1,1)$};
    \node (4) at (0,-4.5) {$\tiny\yng(2,1)$};
    \node (5) at (-2,-4.5) {$\tiny\yng(3,0)$};
    \node (6) at (-1,-6){$\tiny\yng(3,1)$};
    \node (7) at (1,-6){$\tiny\yng(2,2)$};
    \node (8) at (0,-7.5){$\tiny\yng(3,2)$};
    \node (9) at (0,-9){$\tiny\yng(3,3)$};
    \draw [edge] (0) to (1);
    \draw [edge] (1) to (2);
    \draw [edge] (1) to (3);
    \draw [edge] (2) to (4);
    \draw [edge] (3) to (4);
    \draw [edge] (2) to (5);
    \draw [edge] (5) to (6);
    \draw [edge] (4) to (6);
    \draw [edge] (4) to (7);
    \draw [edge] (6) to (8);
    \draw [edge] (7) to (8);
    \draw [edge] (8) to (9);
    \draw [edge] (9) to [bend left=100] (2); 
    \draw [edge] (8) to [bend right=100] (1);
    \draw [edge] (6) to [bend left=100] (0); 
\end{tikzpicture}
\caption{Quantum Bruhat graph for $\Gr(2,5)$}\label{fig1b}
\end{subfigure}%

\caption{Quantum Bruhat graphs}
\end{figure}
\label{graphs}

We are now ready to state our main theorem.
\begin{thm} \label{mainthm}
The largest real eigenvalue $\delta_0$ of $\hat c_1$ for $\Gr(k,n)$ is greater than or equal to $ \dim \Gr(k,n) +1 = k(n-k)+1$ with equality if and only if $\Gr(k,n)=\mathbb{P}^{n-1}$.
\end{thm}

We will give two proofs. One is due to the authors and the other is given to us by an anonymous referee. The value $\delta_0$ is explicitly computed by Rietsch in \cite{Riet} and will be the key ingredient to the first proof. The proof uses calculus in conjunction with the explicit computation of $\delta_0$. To the best of the author's knowledge, the study of $\delta_0$ and Conjecture $\mathcal{O}$ has occurred on a case-by-case basis by considering fixed Fano varieties. Our approach deviates from the case-by-case approach and uses the two parameters $k$ and $n$ to prove the main result. The second proof by the anonymous referee is shorter and makes use of elementary inequalities. The proof will be presented in the final section of the manuscript.

{\em Acknowledgements.} The third named author thanks Leonardo Mihalcea for useful discussions. We thank the anonymous referees for their comments and an alternative proof.

\section{Preliminaries}
 We follow \cite{Riet} to state the preliminary results. The eigenvalue $\delta_0$ of $\hat c_1$ is computed by evaluating a Schur function at a particular $k$-tuple. There are many definitions of Schur polynomials; we will use the version obtained from complete symmetric polynomials. Let $h_0,h_1, h_2, h_3, \cdots$ denote the complete symmetric polynomials in $k$ variables where $\deg h_i=i$ for $i \geq 0$ and let $h_i=0$ for $i <0$. Given $\lambda \in \Lambda$, the Schur polynomial $S_\lambda$ is given by
\[S_\lambda = \begin{vmatrix} 
h_{\lambda_1} & h_{\lambda_1+1} & \cdots & h_{\lambda_1+k-1} \\
h_{\lambda_2-1} & h_{\lambda_2} & \cdots & h_{\lambda_2+k-2} \\
 &  & \ddots &  \\
 h_{\lambda_k-k+1} & \cdots & \cdots & h_{\lambda_k} \\
\end{vmatrix}. \]
We refer the reader to \cite{Macdonald} for further details on Schur functions.
We now give a useful definition for the computation of $\delta_0$.
\begin{defn}[The set $\mathcal{I}_{k,n}$]
We fix the primitive $n$-th root of unity $\zeta=e^{\frac{2\pi i}{n}}$. Let $\zeta^I:=(\zeta^{i_1}, \cdots, \zeta^{i_k})$ be an unordered $k$-tuple of distinct $n$-th roots of $(-1)^{k+1}$. Then $I=(i_1,\cdots,i_k)$ maybe chosen uniquely such that $-\frac{k-1}{2} \leq i_1 < i_2<\cdots<i_k \leq n-\frac{k+1}{2}$ and the $i_k$'s are all integers (respectively half-integers) if $k$ is odd (resp. even). Denote the set of all $k$-tuples $I$ by $\mathcal{I}_{k,n}$.
\end{defn} 
It is known that $\mbox{QH}^\bullet(\Gr(k,n))$ is a finite dimensional vector space with basis given by (the restriction of) the Schubert basis $\{ \sigma_\lambda \}_{\lambda \in \Lambda}$. Rietsch presents another basis given by $\sigma_I:= \sum _{\nu \in \Lambda} \overline{S_{\nu}(\zeta^I)} \sigma_{\nu}$ for $I \in \mathcal{I}_{k,n}$. Consider the multiplication operator $[\sigma_\lambda]: \sigma \mapsto \sigma_\lambda \star \sigma$ on $\mbox{QH}^{\bullet}(\Gr(k,n))$. It follows that 
\[\sigma_\lambda \star \sigma_{I}=S_{\lambda}(\zeta^I)\sigma_I \]
and $\sigma_I$ is an eigenvector of the multiplication operator $[\sigma_\lambda]$ with eigenvalue $S_\lambda(\zeta^I)$. In particular, $\hat c_1 = n[\sigma_{(1)}]$ and $\delta_0=n \max \{ |S_{(1)}(\zeta ^I)| I \in \mathcal{I}_{k,n} \}.$
This leads us to a Lemma by Rietsch which explicitly computes $\delta_0$.
\begin{lemma}[\cite{Riet} Proposition 11.1]
Let $I_0=(-\frac{k-1}{2}, \cdots, \frac{k-1}{2})$ and $I \in \mathcal{I}_{k,n}$ any other elements, and let $\zeta=e^{\frac{2 \pi i}{n}}$ as before. Then for any partitions $\lambda \in \Lambda$, \[|S_{\lambda}(\zeta^I)| \leq S_\lambda(\zeta^{I_0}). \]
In particular, $\delta_0=n S_\lambda(\zeta^{I_0}) \in \mathbb{R}.$
\end{lemma}
\begin{remark} 
Another formulation is $\delta_0=n \frac{\sin(\pi k/n)}{\sin(\pi/n)}$. See \cite[Lemma 9.21]{GaIr}. This is the formulation that the anonymous referee uses in their proof.
\end{remark}
The number $\delta_0=n S_{(1)}(\zeta^{I_0})$ is computed directly for $\Gr(k,n)$. For notation, let $\delta_0^k(n)=\delta_0$ for $\Gr(k,n)$.  If $k$ is even then we have that
\begin{eqnarray*}
\delta^k_0(n)
&=&2n \sum_{j=1}^{\frac{k}{2}}  \cos\left( \frac{(k-(2j-1)) \pi }{n} \right).
\end{eqnarray*}
If $k$ is odd then we have that
\begin{eqnarray*}
\delta^k_0(n)=n+ 2n \sum_{j=1}^{\frac{k-1}{2}}  \cos\left( \frac{(k-(2j-1)) \pi }{n} \right).
\end{eqnarray*}
To prove Galkin's bound for $\Gr(k,n)$ it suffices to show that $\delta^k_0(n)-k(n-k)-1 \geq 0$ for all $k \in \mathbb{Z}^+$ with equality if and only if $k=1$ or $k=n-1$. Thus it is natural to consider the family of real valued functions $F^k(x): \mathbb{R} \rightarrow \mathbb{R}$, indexed by $k \in \mathbb{Z}^+$, that are defined by $F^k(x)=\delta_0^k(x)-k(x-k)-1$ where

\begin{eqnarray*}
\displaystyle \delta_0^{k}(x)= \begin{cases} 
     \displaystyle 2x \sum_{j=1}^{\frac{k}{2}}  \cos\left( \frac{(k-(2j-1)) \pi }{x} \right), & k \mbox{ even} \\
    \displaystyle  x+ 2x \sum_{j=1}^{\frac{k-1}{2}}  \cos\left( \frac{(k-(2j-1)) \pi }{x} \right), & k \mbox{ odd.}
   \end{cases}
\end{eqnarray*}
The next Lemma is a technical Lemma that allows us to reduce the number of cases that need to be checked for Galkin's lower bound to hold.
\begin{lemma} \label{sufflemma}
Let $k \geq 4$. It suffices to show that Galkin's lower bound holds for $\Gr(k,n)$ where $n \geq 2(k-1)$.
\end{lemma}
\begin{proof}
Consider $\Gr(k,n)$ where $k+1 \leq n \leq 2(k-1).$ Then $\Gr(k,n)$ is isomorphic to $\Gr(\ell,n)$ where $\ell=n-k$. The Lemma follows from the observation that $2k+2>2k-2$ and $2k+2 \geq n \Leftrightarrow n \geq 2(l-1)$.
\end{proof}

\section{Proof of the Main Theorem}

\subsection{Calculus on the functions $F^k(x)$}
In this section, we use calculus to explore features of the function $F^k$ to further understand the number $\delta_0^k(n)$. We will prove Proposition \ref{propinc} by using Lemma \ref{techlemma} and Lemma \ref{lemmacalc}.

\begin{prop}\label{propinc}
Let $k \geq 2$. The function $F^k$ is increasing for $x \in (2(k-1),\infty).$
\end{prop}
The proof of the first Lemma is reserved until Section \ref{LEMMA}. The proof involves several applications of the Mean Value Theorem.

\begin{lemma} \label{techlemma}
Let $f:\mathbb{R} \rightarrow \mathbb{R}$ be a second differentiable function where $\displaystyle \lim_{x \rightarrow \infty}f(x)=L\in \mathbb{R}$. If $f''(x)<0$ for all $x \in (a,\infty)$ then $f'(x)>0$ for all $x \in (a,\infty)$.
\end{lemma}
The next Lemma follows by nice properties of the function $F^k.$ The ability to calculate the $\displaystyle \lim_{x \rightarrow \infty} F^k(x)$ is a straight forward application of l'H\^{o}pital's rule after slightly reorganizing the expression. The authors anticipate that similar asymptotic behavior will occur in at least the submaximal isotropic Grassmannians in types $B$, $C$, and $D$. The second derivative of $F^k(x)$ takes on a nice form since the $\sin$ functions that appear in the first derivative cancel out in the calculation of the second derivative.

\begin{lemma} \label{lemmacalc}
 
We have that 
\begin{enumerate}
\item $\displaystyle \lim_{x \rightarrow \infty} F^k(x)=k^2-1$;
\item Let $k \geq 2$. The function $F^k(x)$ is concave down on the interval $(2(k-1), \infty)$;
\item The equality $F^k(2(k-1)) = F^{k-2}(2(k-1))$ holds for $k \geq 3$.
\end{enumerate}
\end{lemma}
\begin{proof}
For the first part observe that l'H\^{o}pital's rule implies that \[ \lim_{x \rightarrow \infty} 2x\left( \cos\left( \frac{(k-(2j-1)) \pi }{x} \right)-1\right)=0 \mbox{ for any integer }j, \mbox{ for }1 \leq j \leq \frac{k-1}{2}.\]
If $k$ is even then
\begin{eqnarray*}
\lim_{x\rightarrow \infty} F^k(x) &=& \lim_{x \rightarrow \infty} \left( \sum_{j=1}^{\frac{k}{2}}  2x\left( \cos\left( \frac{(k-(2j-1)) \pi }{x} \right)-1\right) \right)+k^2-1\\
&=&k^2-1.
\end{eqnarray*}
In a similar fashion, if $k$ is odd, it follows that $\displaystyle\lim_{x\to \infty}F^k(x)=k^2-1$ by a simple calculation.

For the second part a direct calculation shows that
\begin{eqnarray*}
\displaystyle \frac{d^2}{dx^2}F^k(x)&=&-\sum_{j=1}^{\frac{k}{2}}  \left(\frac{(k-(2j-1))^2 \pi^2}{x^3}\right)\cos\left( \frac{(k-(2j-1)) \pi }{x} \right).
\end{eqnarray*}
Since $0 < k+1-2j \le k-1$ for $j=1,2,\cdots,\frac{k}{2}$ then
\begin{eqnarray*}
0 < \frac{(k-(2j-1))\pi}{2(k-1)}\le \frac{\pi}{2}.
\end{eqnarray*}
It follows that $0 < \frac{(k-(2j-1))\pi}{x}< \frac{\pi}{2}$ for all $x \in (2(k-1),\infty)$. Thus $0< \cos\left( \frac{(k-(2j-1)) \pi }{x} \right) < 1$ for all $x\in (2(k-1),\infty)$. Therefore $\frac{d^2}{dx^2}F^k(x)<0$ for all $x \in (2(k-1),\infty)$.
The odd case can be shown in a similar fashion.

The third part of the proof follows from the isomorphism $\Gr(k,2(k-1)) \cong \Gr(k-2,2(k-1))$.
\end{proof}
\begin{remark}
Part 3) of the proof can easily be proved analytically since the equation $F^k(x)=F^{k-2}(x)$ simplifies to
\[x\left({\cos\left({\frac{(k-1)\pi}{x}}\right)}\right)+2(k-1)-x=0. \] The number $x=2(k-1)$ is a root.
\end{remark}

Proposition \ref{propinc} follows from Lemmas \ref{techlemma} and \ref{lemmacalc}. 

\begin{figure}[h!]
    \centering
    \caption{The function $F^k(x)$ plotted for $k=2,4,6,8$.}
    \label{fig:my_label}
\begin{tikzpicture}[scale=.5]
\begin{axis}[
    axis lines = left,
    xlabel = $x$,
    ylabel = {$F^k$},
    xmin=3, xmax=100,
    ymin=0,ymax=70,
    width=.9\textwidth,
    height=0.6\textwidth,
    legend style={at={(axis cs:10,40)},anchor=south west}
]
\addplot [
    domain=3:100, 
    samples=100, 
    color=red,
]
{2*x*(cos(deg(((2-(2-1))*(3.14))/(x))))-2*(x-2)-1};
\addlegendentry{$F^{2}(x)$}
\addplot [
    domain=5:100, 
    samples=100, 
    color=blue,
]
{2*x*cos(deg(((4-(2-1))*(3.14))/(x)))+2*x*cos(deg(((4-(4-1))*(3.14))/(x)))- 4*(x-4)-1};
\addlegendentry{$F^{4}(x)$}

\addplot [
    domain=7:100, 
    samples=100, 
    color=green,
]
{2*x*cos(deg(((6-(2-1))*(3.14))/(x)))+2*x*cos(deg(((6-(4-1))*(3.14))/(x)))+2*x*cos(deg(((6-(6-1))*(3.14))/(x)))- 6*(x-6)-1};
\addlegendentry{$F^{6}(x)$}
\addplot [
    domain=9:100, 
    samples=100, 
    color=orange,
]
{2*x*cos(deg(((8-(2-1))*(3.14))/(x)))+2*x*cos(deg(((8-(4-1))*(3.14))/(x)))+2*x*cos(deg(((8-(6-1))*(3.14))/(x)))+2*x*cos(deg(((8-(8-1))*(3.14))/(x)))- 8*(x-8)-1};
\addlegendentry{$F^{8}(x)$}

\end{axis}
\end{tikzpicture}
\end{figure}
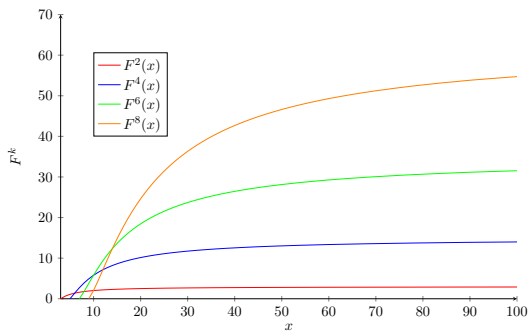

\subsection{}
The next proposition follows by using induction and the property $F^k$ is increasing for $x \in (2(k-1),\infty)$ when $k \geq 2$.
\begin{prop} \label{mainprop}
Galkin's lower bound holds for 
\begin{enumerate}
\item $\Gr(1,n)$ (which is isomorphic to $\mathbb{P}^{n-1})$ where $n \geq 2$;
\item $\Gr(2,n)$ where $n \geq 3$;
\item $\Gr(3,n)$ where $n \geq 4$;
\item $\Gr(k,n)$ where $n \geq 2(k-1)$ and $k\geq 4$.
\end{enumerate}
\end{prop}
\begin{proof}
The first case is known and can easily be seen in this context by evaluating $F^1(x)=0$ for $x \geq 2$.

For the second case where $k=2$ observe that $\Gr(2,3)$ is isomorphic to $\mathbb{P}^2$ and a direct calculation shows that $F^2(3)=0$. Since $F^2(x)$ is increasing for $x>2$ it follows that $F^2(x)>0$ for $x>3$. Observe that $\Gr(2,n)$ is not a projective space for $n>3$. Case (2) follows.

 For the third case where $k=3$ observe that $\Gr(3,4)$ is isomorphic to $\mathbb{P}^3$ and a direct calculation shows that $F^3(4)=0$. Since $F^3(x)$ is increasing for $x>4$ it follows that $F^3(x)>0$ for $x>4$. Observe that $\Gr(3,n)$ is not a projective space for $n>4$. Case (3) follows.

For the fourth case, it suffices to show that $F^k(2(k-1))>0$ since $F^k(x)$ is increasing for $x >2(k-1)$. We will use induction to show this inequality. A direct calculation shows that $F^4(6)=6\sqrt{3}-9>0.$ Assume that $F^k(2(k-1))>0$ where $k>4$ is even. Since $F^k(x)$ is increasing for $x \geq 2(k-1)$ then we have that \[F^{k+2}(2((k+2)-1))=F^{k}(2((k+2)-1))>F^k(2(k-1))>0. \] Observe that $\Gr(k,n)$ where $k \geq 4$ and $n \geq 2(k-1)$ is not a projective space. The even case follows. The case where $k \geq 5$ is odd follows by implementing a similar strategy.
\end{proof}

Theorem \ref{mainthm} follows from Lemma \ref{sufflemma} and Proposition \ref{mainprop}.

\begin{figure}[h!]
    \centering
    \caption{Illustration of Proposition \ref{mainprop}}
    \label{fig:my_label}

\begin{tikzpicture}[scale=.5]
\begin{axis}[
    axis lines = left,
    xlabel = $x$,
    ylabel = {$F^k$},
    xmin=3, xmax=20,
    ymin=0,ymax=20,
    width=.9\textwidth,
    height=0.6\textwidth,
    legend style={at={(axis cs:6,13)},anchor=south west}
]
\addplot [
    domain=3:30, 
    samples=100, 
    color=red,
]
{2*x*(cos(deg(((2-(2-1))*(3.14))/(x))))-2*(x-2)-1};
\addlegendentry{$F^{2}(x)$}
\addplot [
    domain=5:30, 
    samples=100, 
    color=blue,
]
{2*x*cos(deg(((4-(2-1))*(3.14))/(x)))+2*x*cos(deg(((4-(4-1))*(3.14))/(x)))- 4*(x-4)-1};
\addlegendentry{$F^{4}(x)$}

\addplot [
    domain=7:30, 
    samples=100, 
    color=green,
]
{2*x*cos(deg(((6-(2-1))*(3.14))/(x)))+2*x*cos(deg(((6-(4-1))*(3.14))/(x)))+2*x*cos(deg(((6-(6-1))*(3.14))/(x)))- 6*(x-6)-1};
\addlegendentry{$F^{6}(x)$}
\addplot [
    domain=9:30, 
    samples=100, 
    color=orange,
]
{2*x*cos(deg(((8-(2-1))*(3.14))/(x)))+2*x*cos(deg(((8-(4-1))*(3.14))/(x)))+2*x*cos(deg(((8-(6-1))*(3.14))/(x)))+2*x*cos(deg(((8-(8-1))*(3.14))/(x)))- 8*(x-8)-1};
\addlegendentry{$F^{8}(x)$}
\addplot[smooth,mark=*,black] plot coordinates {
        (6,1.39389713287)
    };
    \addlegendentry{(6,$F^2(6))=(6,F^4(6)$)}
    
\addplot[smooth,mark=*,black] plot coordinates {
        (10,5.8)
    };
\addlegendentry{(10,$F^4(10))=(10,F^6(10)$)}
\addplot[smooth,mark=*,black] plot coordinates {
        (14,12.38)
    };
\addlegendentry{(14,$F^6(14))=(14,F^8(14)$)}
    
\end{axis}
\end{tikzpicture}
\end{figure}
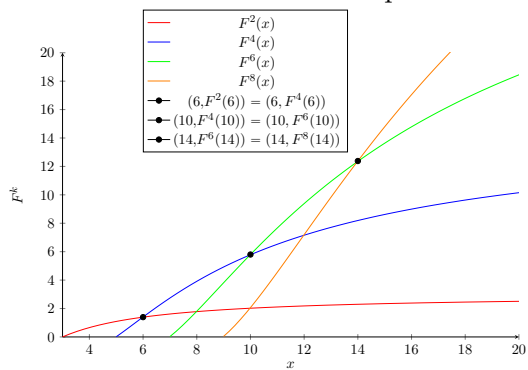

\section{Proof of Lemma \ref{techlemma}} \label{LEMMA}
In this section we prove Lemma \ref{techlemma}. The authors were unable to find a source and we expect that this Lemma is well-known. We first recall a version of the Mean Value Theorem for the second derivative and prove Lemma $\ref{techlemma}$ by its contrapositive statement in Lemma \ref{LEMMA2}.

\begin{lemma} \label{MVTSD}
Let $f:\mathbb{R} \rightarrow \mathbb{R}$ be a second differentiable function. If $f'(a)\leq 0$ and $f'(b) \geq 0$ for some $a,b \in \mathbb{R}$ where $a<b$ then there exists a number $c \in (a,b)$ where $f''(c)\geq 0$.
\end{lemma}




\begin{lemma} \label{LEMMA2}
Let $f:\mathbb{R} \rightarrow \mathbb{R}$ be a second differentiable function where $\displaystyle \lim_{x \rightarrow \infty}f(x)=L\in \mathbb{R}$. If $f'(b) \leq 0$ for some $b \in (a, \infty)$ then there exists a number $c \in (a, \infty)$ where $f''(c) \geq 0$.
\end{lemma}

\begin{proof}
Case (I) Consider the case that $f(b) < L$. Since $f$ is continuous and $y=L$ is a horizontal asymptote of $f$, there exists a $c \in (b,\infty)$ where $f(b)<f(c)<L$. By the Mean Value Theorem there exists a $d \in (b,c)$ where $f'(d)\ge 0.$ By Lemma \ref{MVTSD}, there exists $e \in (b,d) \subset (a,\infty)$ where $f''(e) \geq 0$. The result follows.

Case (II) Consider the case that $f(b) \geq L$.
Subcase (1) If $f(b)=L$ and $f(x)=L$ is the constant function then we are done. 

Subcase (2): Suppose that there exists a $c\in (b, \infty)$ where $f(c)<L.$ By the Mean Value theorem there exists $d \in (b,c)$ where $f'(d)=\frac{f(c)-f(b)}{c-b}<0.$ Since $f(c)<L$ and $y=L$ is a horizontal asymptote of $f$, there exists $e \in (c,\infty)$ where $f(c)<f(e)<L.$ By The Mean Value Theorem there exists $h \in (c,e)$ where $f'(h)=\frac{f(e)-f(c)}{e-c}>0$. The result follows from Lemma \ref{MVTSD} since $f'(d)<0$ and $f'(c)>0.$

Subcase (3): Suppose that $f(x) \geq L$ for all $x \in (b,\infty)$. We may assume that there is a $d \in (b,\infty)$ where $f(d)>L.$

Sub-subcase (i): Suppose that there exists $e \in (d,\infty)$ where $f(e)=L$. Choose $h \in (e, \infty)$. Then by the Mean Value Theorem there exits $k \in (e,h)$ where $f'(k)=\frac{f(h)-f(e)}{h-e} \geq 0$. The result follows from Lemma \ref{MVTSD}.

Sub-subcase (ii): Since $y=L$ is a horizontal asymptote of $f$ then there exists a $e \in (d,\infty)$ where $L<f(e)<f(d)$. By the Mean Value Theorem there exists an $h \in (d,e)$ where $f'(h)=\frac{f(e)-f(d)}{e-d}<0.$ Observe that \[ \lim_{x \rightarrow \infty} \frac{f(x)-f(e)}{x-e}=0. \] Thus one can choose $k \in (e,\infty)$ where $\frac{f(k)-f(e)}{k-e} \geq f'(h).$ In particular, there exists an $i \in (e,k)$ where $f'(i)=\frac{f(k)-f(e)}{k-e} \geq f'(h)$ by the Mean Value Theorem. By the Mean Value Theorem, again, there exists a $l \in (h,i) \subset (a, \infty)$ where $f''(l)=\frac{f'(i)-f'(h)}{i-h} \geq 0$. This completes the proof.
\end{proof}
Lemma \ref{techlemma} follows.

\section{A second proof}
The authors would like to thank the anonymous referee for sharing the following alternative proof. Recall that $\displaystyle \delta_0(x)=n\frac{\sin(\pi x/n)}{ \sin(\pi/n)}$. We begin with a Lemma.

\begin{lemma} \label{lemma}
Let $n \geq 6$ be a positive integer. Then $\displaystyle n\frac{\sin(\pi x/n)}{ \sin(\pi/n)} \geq x(n-x)+1$ for $x \in [3, n/2]$.
\end{lemma}

\begin{proof}
We use the elementary inequalities $\displaystyle x-\frac{x^3}{6}=x\left(1-\frac{x^2}{6} \right) \leq \sin(x) \leq x$ for $x \geq 0$ to see that \[ \displaystyle n\frac{\sin(\pi x/n)}{ \sin(\pi/n)} \geq \frac{n(\frac{\pi x}{n})\left(1-\frac{(\frac{\pi x}{n})^2}{6} \right)}{\frac{\pi}{n}} = nx-\frac{\pi^2x^3}{6n}. \]
Now for $3 \leq x \leq \frac{n}{2}$ we have $x^3 \leq (n/2)x^2$ so we have that \[ \displaystyle nx-\frac{\pi^2x^3}{6n} \geq nx-\frac{\pi^2x^2}{12}.\]
To prove the lemma, it suffices to show that $\displaystyle nx-\frac{\pi^2x^2}{12} \geq nx-x^2+1$ for $x \in [3,n/2]$. Equivalently, we must show that $x^2 \geq \left(1-\frac{\pi^2}{12} \right)^{-1}$. Using the fact that $\pi^2<10$, we see that $\left(1-\frac{\pi^2}{12}\right)^{-1}<6<x^2$ for $x \geq 3$. The inequality holds for $x \in [3,n/2]$. The result follows.
\end{proof}

\begin{prop}
Let $n \geq 4$ be a positive integer and let $k \in \{2,3, \cdots, \lfloor n/2 \rfloor \}$. Then $n\frac{\sin(\pi k/n)}{ \sin(\pi/n)} >k(n-k)+1$.
\end{prop}

\begin{proof}
For $k \geq 3$ and consequently $n \geq 6$, the result follows from Lemma \ref{lemma}. It suffices to consider the case when $k=2$. By the double angle formula, we have $\sin(2 \pi/n)=2 \cos(\pi/n)\sin(\pi/n)$, so the case $k=2$ reduces to the inequality \[\displaystyle 2n \cos(\pi/n) \geq 2n-3. \] Using the elementary inequality $\displaystyle \cos(x) \geq 1-\frac{x^2}{2}$ for $x \geq 0$, we have that \[\displaystyle 2n \cos \left(\pi/n \right) \geq 2n- \frac{2n \left (\frac{\pi}{n} \right)^2}{2}=2n-\frac{\pi^2}{n}. \] Since $n \geq 4$ and $\pi^2<10$ we see $2n-\frac{\pi^2}{n}>2n-3$. The result follows.
\end{proof}

Since $\Gr(k,n)$ is isomorphic to $\Gr(n-k,n)$ it is suffices to show that Galkin's conjectured lower bound holds for $k=1,2, \cdots, \lfloor n/2 \rfloor$ with equality only when $k=1$. This follows immediately from the above inequalities for $n \geq 4$. The bound obviously holds for the $n=2$ and $n=3$ cases. Theorem \ref{mainthm} follows.


\bibliography{Oconj}
\bibliographystyle{halpha}
\end{document}